\documentclass[11pt,a4paper]{article}

\usepackage{amsmath,amsfonts,amssymb,amsthm}
\usepackage{bbm,mathrsfs, bm, mathtools}

\title{Separation bodies: a conceptual dual to floating bodies}

\author{Rolf Schneider}
\date{}
\newcommand{\Sd}{{\mathbb S}^{d-1}}
\newcommand{\R}{{\mathbb R}}

\newcommand{\bP}{{\mathbb P}}

\newcommand{\Rd}{{\mathbb R}^d}
\newcommand{\N}{{\mathbb N}}

\newcommand{\Ha}{\mathcal{H}}

\newcommand{\D}{{\rm d}}

\newcommand{\bE}{{\mathbb E}}

  \renewcommand{\exp}{{\rm exp}\,}

\newtheorem{theorem}{Theorem}
\newtheorem{lemma}{Lemma}
\newtheorem{definition}{Definition}

\begin{document}
\maketitle

\begin{abstract}
Let $K$ be a convex body in Euclidean space ${\mathbb R}^d$, and let a translation invariant, locally finite Borel measure on the space of hyperplanes in ${\mathbb R}^d$ be given. For $\delta\ge 0$, we consider the set of all points $x$ for which the set of hyperplanes separating $K$ and $x$ has measure at most $\delta$. This defines the separation body of $K$, with respect to the given measure and the parameter $\delta$. Separation bodies are meant as conceptual duals to floating bodies, and they are expected to play a role in the investigation of random polytopes generated as intersections of random halfspaces, in a similar way that floating bodies are useful for studying convex hulls of random points. After discussing some elementary properties of separation bodies, we carry out first examples to this effect.\\[1mm]
\end{abstract}

\section{Introduction}\label{sec1}

For a convex body $K$ (a compact, convex set with interior points) in Euclidean space $\R^d$ and for a sufficiently small parameter $\delta\ge 0$, the {\em floating} body $K_\delta$ is defined as follows. According to the approach of B\'{a}r\'{a}ny and Larman \cite{BL88}, for $x\in K$ let
$$ v(x):= \min\{V(K\cap H^+):x\in H^+,\, H^+\mbox{ a closed halfspace}\},$$
where $V$ denotes the volume, and define
$$ K(v\ge \delta):= \{x\in K:v(x)\ge \delta\}.$$ 
By the approach of Sch\"utt and Werner \cite{SW90}, for $u\in \Sd$ (the unit sphere of $\R^d$) and for sufficiently small $\delta\ge 0$ there is a unique number $t(u,\delta)\ge 0$ such that
$$ V(K\cap H^+(u,h(K,u)-t(u,\delta)))=\delta,$$
where $h(K,\cdot)$ is the support function of $K$ and 
$$ H^+(u,\tau):= \{x\in \R^d: \langle x,u\rangle\ge \tau\},\qquad H^-(u,\tau):= \{x\in \R^d: \langle x,u\rangle\le \tau\},$$
with $\langle\cdot\,,\cdot\rangle$ denoting the scalar product of $\Rd$. Define the closed convex set
$$ K_\delta:= \bigcap_{u\in \Sd} H^-(u,h(K,u)-t(u,\delta)).$$
It is easy to see that $ K(v\ge\delta)=K_\delta$. For some historical remarks, we refer to \cite[Section 10.6]{Sch14}.

B\'{a}r\'{a}ny and Larman have introduced the floating body (though not under this name) as a tool in stochastic geometry, mainly for treating certain questions about convex hulls of independent uniform random points in a given convex body. For this, and for expansions of this approach, we refer to \cite{Bar89}, \cite{Bar07}, \cite{Bar08}, \cite{BL88}, \cite{BV93}, \cite{Fre13}. The starting point of Sch\"utt and Werner was a formula of Blaschke for introducing the affine surface area of a three-dimensional convex body with analytic boundary, which was extended to higher dimensions and weaker smoothness assumptions by Leichtwei{\ss} (references are in  \cite[Section 10.6]{Sch14}). In \cite{SW90} this was further extended to general convex bodies. The introduction of convex floating bodies by these authors was the beginning of a long series of investigations on relations between floating bodies and affine surface area, general studies of floating bodies, their relation to approximations of convex bodies, extended notions of affine surface area, analogues of floating bodies, for example with volume replaced by surface area; see \cite{BLW18}, \cite{BSW18}, \cite{BW17}, \cite{BW18}, \cite{Fre12}, \cite{LSW19}, \cite{NSW19}, \cite{Sch91}, \cite{Sch99}, \cite{Sch02}, \cite{SW92}, \cite{SW94}, \cite{SW04}, \cite{Wer96}, \cite{Wer99}, \cite{Wer02}, \cite{Wer06}, \cite{Wer07}, \cite{WY08}, \cite{WY10}, \cite{WY11}.

We emphasize one of the analogues of the floating body. For $x\in \R^d$, let
$$ K^x:={\rm conv}(K\cup\{x\}).$$
For $\delta\ge 0$, Werner \cite{Wer94} defined the {\em illumination body} of $K$ with parameter $\delta$ by
\begin{equation}\label{1.1} 
K^\delta:= \{x\in\R^d: V(K^x)-V(K)\le\delta\}.
\end{equation}
Since this construction is invariant under volume preserving affine transformations, also the illumination body can be used to yield, by approximation, the affine surface area. 

Mordhorst and Werner \cite{MW17, MW17b} investigated in how far floating bodies and illumination bodies can be considered as approximately dual to each other. We quote from \cite{MW17}: ``We investigate a duality relation between floating and illumination bodies. The definition of these two bodies suggests that the polar of the floating body should be similar to the illumination body of the polar.''

The aim of the following is to introduce a modification of the illumination bodies of $K$, which we call separation bodies. They are no longer affinely related to $K$ (and hence cannot be expected to be relevant for the affine surface area), but can with good reasons be considered as a more proper dual to floating bodies. Here `dual' does not point to an exact duality, but is meant in a conceptual sense. If the separation body is constructed with a motion invariant measure, one may conjecture, modifying the citation above,  that the polar of the floating body is `more similar' (expressed by a better order of approximation) to the separation body of the polar than to the illumination body of the polar. This, however, has not yet been verified.

The starting point of the present note is the role of the floating body in stochastic geo\-metry. It was introduced there to support the investigation of random polytopes which are generated as convex hulls of random points in a given convex body. The dual generation of random polytopes, namely as the intersection of random halfspaces containing a given convex body, calls for a conceptually dual construction. After introducing separation bodies in the next section, we discuss first applications of them in Section \ref{sec3}. The explanation of our main results is postponed to that section, since some preparations are required.

\section{Separation bodies: definition and elementary properties}\label{sec2}

Let $K\subset\R^d$ be a convex body with $o\in{\rm int}\,K$. We recall that its polar body is defined by
$$ K^\circ = \{x\in\R^d: \langle x,y\rangle \le 1 \mbox{ for all }y\in K\}.$$

In the definition of floating bodies, the volumes of nonempty caps $K\cap H^+$, where $H^+$ is a closed halfspace, play a role. What corresponds to such a cap under duality? One possible effect of duality is that points and hyperplanes interchange their roles. Let $\Ha$ denote the space of hyperplanes in $\R^d$, with its usual topology. For a set $M\in\R^d$, we write 
$$\Ha_M:=\{H\in\Ha:H\cap M\not=\emptyset\}.$$ 
For a convex body $K$ and for $x\in\R^d$, we denote by $\Ha(K|x)$ the set of hyperplanes that weakly separate $K$ and $x$. Hyperplanes will often be written in the form
$$ H(u,\tau)=\{x\in \R^d:\langle x,u\rangle=\tau\}\quad\mbox{with }u\in \Sd,\,\tau\in \R. $$
We define a map $\eta:\R^d\setminus\{o\}\to\Ha\setminus \Ha_{\{o\}}$ by
$$ \eta(ru)=H(u,r^{-1}) \quad\mbox{for }u\in\Sd,\,r>0.$$
Now let $K\cap H^+$ be a full-dimensional cap of $K$, where $H^+$ is a closed halfspace bounded by a hyperplane $H$  and not containing $o$. Let $x:= \eta^{-1}(H)$. Then $x\notin K^\circ$, and $\eta$ maps the cap $K\cap H^+$ bijectively onto the set $\Ha(K^\circ|x)$ of hyperplanes weakly separating $K^\circ$ and $x$. (The easy proof can be found in \cite{HS19}.) 

Instead of the Lebesgue measure on sets of points (which is up to a factor the unique translation invariant, locally finite Borel measure), we consider now a translation invariant and locally finite Borel measure on the space $\Ha$ of hyperplanes (and not the image measure of the Lebesgue measure under the map $\eta$, because our goal is a translation invariant construction). Any such measure is a constant multiple of a measure of the form
\begin{equation}\label{2.1} 
\nu= \int_{\Sd} \int_{\R} {\mathbbm 1}\{H(u,\tau)\in\cdot\}\,\D\tau\,\varphi(\D u),
\end{equation}
with an even probability measure $\varphi$ on the sphere $\Sd$ (see \cite[(4.33)]{Sch14}).

\vspace{2mm}

\noindent {\bf Assumption.} We assume in the following that an even probability measure $\varphi$ on $\Sd$ is given, which is not concentrated on a great subsphere.

\vspace{2mm}

With $\nu$ given by (\ref{2.1}), we define
\begin{equation}\label{2.1b} 
m(K,x):= \nu(\Ha(K|x)).
\end{equation}
For any convex body $K$ it follows immediately from (\ref{2.1}) that
$$ \nu(\Ha_K) = 2\Phi(K),$$
where the functional $\Phi$ is defined by 
\begin{equation}\label{2.1a}      
\Phi(K)= \int_{\Sd} h(K,u)\,\varphi(\D u).
\end{equation}
(We recall that $h(K,\cdot)$ denotes the support function of $K$.) Therefore,
$$ m(K,x) = 2[\Phi(K^x)-\Phi(K)].$$

\begin{definition}
For $\delta\ge 0$, we define
$$ K[\varphi,\delta]:=\{x\in\R^d: m(K,x)\le\delta\}.$$
We call $K[\varphi,\delta]$ the {\em separation body} of $K$ with respect to $\varphi$ with parameter $\delta$.
\end{definition}

In the special case  where $\varphi=\sigma$, the normalized spherical Lebesgue measure on $\Sd$, we get $2\Phi(K)=W(K)$, with $W$ denoting the mean width. Thus, we have
$$ K[\delta]:= K[\sigma,\delta] =\{x\in\R^d: W(K^x)-W(K)\le \delta\},$$
which is of the type (\ref{1.1}), with the volume replaced by the mean width. The bodies $K[\delta]$, which appear in \cite{BS10}, were extended in a different way by Jenkinson and Werner \cite{JW14} (compare their relations (9), (10), (11)). Their extensions, which are not necessarily convex, were used to establish a new relationship between convex geometric analysis and information theory. 

The following interesting observations are due to Olaf Mordhorst (private communication). For a general convex body $K$, one has
$$ \lim_{\delta\to 0} \frac{V(K[\sigma,\delta])-V(K)}{\delta^{2/(d+1)}} = c_d\int_{\partial K} \kappa^{-1/(d+1)}\,\D {\mathscr H}^{d-1}$$
with a constant $c_d$, where $\kappa$ is the (generalized) Gauss curvature  of $K$ and ${\mathscr H}^{d-1}$ is the $(d-1)$-dimensional Hausdorff measure. Further, suppose that the convex body $K$ is centrally symmetric and sufficiently smooth, and let $\varphi$ be the surface area measure $S_{d-1}(K,\cdot)$, normalized to a probability measure. Then the separation body $K[\varphi,\delta]$ is equal to the illumination body $K^{c\delta}$ for some constant $c>0$ depending on $K$ and the dimension $d$.

Now we state some elementary properties of separation bodies.

\begin{lemma}\label{L1}
The separation body $K[\varphi,\delta]$ is closed and convex. If the support of $\varphi$ is the whole sphere $\Sd$, then $K[\varphi,\delta]$ is strictly convex for $\delta>0$.
\end{lemma}

\begin{proof}
For $x,y\in\R^d$ we have
\begin{eqnarray*}
|m(K,x)-m(K,y)| &=& 2|\Phi(K^x)-\Phi(K^y)|\\
& \le &2\int_{\Sd} |h(K^x,u)-h(K^y,u)|\,\varphi(du)\\
&\le& 2\|x-y\|,
\end{eqnarray*}
since the Hausdorff distance of $K^x$ and $K^y$ is at most $\|x-y\|$.
Thus, $m(K,\cdot)$ is continuous, and hence $K[\varphi,\delta]$ is closed. 

We extend (to more general measures) from \cite{BS10} the proof that the separation body $K[\varphi,\delta]$ is convex. Let $x,y\in K[\varphi,\delta]$ and $\alpha\in[0,1]$. Let $z\in K^{(1-\alpha)x+\alpha y}$. Then there exist $k\in K$ and $\beta\in[0,1]$ with
$$ z= (1-\beta)[(1-\alpha)x+\alpha y]+\beta k.$$
It follows that
$$ z=(1-\alpha)[(1-\beta)x+\beta k] +\alpha[(1-\beta)y+\beta k] \in (1-\alpha)K^x+\alpha K^y.$$
Since $z\in K^{(1-\alpha)x+\alpha y}$ was arbitrary, this shows that
\begin{equation}\label{1} 
K^{(1-\alpha)x+\alpha y} \subseteq (1-\alpha)K^x+\alpha K^y.
\end{equation}
(A similar argument is found in F\'{a}ry and R\'{e}dei \cite{FR50}.) 
Since the functional $\Phi$ is increasing under set inclusion and is  Minkowski linear, we obtain
\begin{eqnarray*}
\Phi(K^{(1-\alpha)x+\alpha y})-\Phi(K) &\le& \Phi((1-\alpha)K^x+\alpha K^y) -\Phi(K)\\
&=& (1-\alpha)[\Phi(K^x)-\Phi(K)]+\alpha[\Phi(K^y)-\Phi(K)].
\end{eqnarray*} 
Hence,
\begin{equation}\label{2.2}  
m(K,(1-\alpha)x+\alpha y)\le (1-\alpha)m(K,x)+\alpha m(K,y)\le\delta
\end{equation}
and thus $(1-\alpha)x+\alpha y\in K[\varphi,\delta]$. This shows that $K[\varphi,\delta]$ is convex.

Now we assume that ${\rm supp}\,\varphi=\Sd$ (where  ${\rm supp}\,\varphi$ denotes the support of the measure $\varphi$). Let $\delta>0$, and let $x,y\in {\rm bd}\, K[\varphi,\delta]$, $x\not=y$. Since $m(K,x)=m(K,y)=\delta>0$, we have $x,y\notin K$. Let $\alpha\in (0,1)$. Then there exists a supporting hyperplane $H(K,u)$ of $K$ such that 
$$ x\in {\rm int}\, H^-(K,u), \quad (1-\alpha)x+\alpha y\in {\rm int}\, H^-(K,u),\quad y\in {\rm int}\, H^+(K,u),$$
hence the support function $h(K,\cdot)$ of $K$ satisfies $h(K^x,u)=h(K^{(1-\alpha)x+\alpha y},u)=h(K,u)$ and $h(K^y,u)>h(K,u)$. This gives
$$ h((1-\alpha)K^x+\alpha K^y,u)> h(K,u)=h(K^{(1-\alpha)x+\alpha y},u).$$
Therefore (\ref{1}) holds with strict inclusion. Since ${\rm supp}\,\varphi=\Sd$, the functional $\Phi$ is strictly increasing under set inclusion. Thus, (\ref{2.2}) holds with strict inequality. This, together with $m(K,x)=m(K,y)$, shows that $(1-\alpha)x+\alpha y\in{\rm int}\,K[\varphi,\delta]$ and thus that the body $K[\varphi,\delta]$ is strictly convex. 
\end{proof}

To give a simple example, we assume that $K$ is a convex polygon in the plane, and that $\varphi$ is normalized spherical Lebesgue measure on ${\mathbb S}^1$. In that case, $\Phi(K)=L(K)/2\pi$, where $L$ denotes the perimeter, and hence $m(K,x)=\frac{1}{\pi}[L(K^x)-L(K)]$. Let $E$ be the union of the affine hulls of all edges of $K$, and let $A$ be a component of ${\mathbb R}^2\setminus E$ different from ${\rm int}\,K$. For all points $x\in A$, the two support lines of $K$ through $x$ touch $K$ at the same two vertices. Therefore, the boundary of $K[\varphi,\delta]$ (for suitable $\delta>0$) inside $A$ is an arc of an ellipse. We see that the boundary of $K[\varphi,\delta]$ is a union of arcs of different ellipses. 

We note that $ m(K,x)=0\quad\mbox{for }x\in K$.
In fact, for $x\in {\rm int}\,K$ we have $\Ha(K|x)=\emptyset$. Further, $\nu(\Ha(K|x))=0$ if $x\in{\rm bd}\, K$, since the set of hyperplanes passing through a fixed point has $\nu$-measure zero. However, the separation body $K[\varphi,0]$ may be strictly larger than $K$. To make this precise, we define
$$ K_\varphi := \bigcap_{u\in {\rm supp}\,\varphi} H^-(K,u),$$
where $H^-(K,u)$ denotes the supporting halfspace of $K$ with outer normal vector $u$. Then we have
\begin{equation}\label{2.3} 
K[\varphi,0]= K_\varphi.
\end{equation}
To prove this, first let $x\in{\rm int}\,K_\varphi$. A unit normal vector of any hyperplane separating $x$ and $K$ does not belong to ${\rm supp}\,\varphi$, hence $m(K,x)=0$. Since both sides of (\ref{2.3}) are closed, this shows that $K_\varphi \subseteq K[\varphi,0]$. Conversely, let $x\in\R^d\setminus K_\varphi$. Let $U\subset\Sd$ be the (open) set of all unit vectors $u$ such that for some number $\tau$ the hyperplane $H(u,\tau)$ strictly separates $x$ and $K_\varphi$. A translate of this hyperplane supports $K_\varphi$ and also separates $x$ and $K_\varphi$, and $u$ or $-u$ is its outer unit normal vector. The set of boundary points of $K_\varphi$ where such a normal vector exists has positive $(d-1)$-dimensional Hausdorff measure, hence $S_{d-1}(K_\varphi,U)>0$, where $S_{d-1}(K_\varphi,\cdot)$ denotes the surface area measure of $K_\varphi$. By Lemma 7.5.1 in \cite{Sch14}, $S_{d-1}(K_\varphi,\Sd\setminus{\rm supp}\,\varphi)=0$. It follows that $U\cap{\rm supp}\,\varphi\not=\emptyset$ and hence that $\varphi(U)>0$. For each $u\in U$ the set of all $\tau$ for which $H(u,\tau)$ strictly separates $x$ and $K_\varphi$ is a non-degenerate interval. Now it follows that $m(K,x)>0$ and hence that $x\notin K[\varphi,0]$. Thus, (\ref{2.3}) is proved.

The previous remark shows that if we want to express the support function of $K[\varphi,\delta]$ in terms of $\delta$, we must restrict ourselves to the support of $\varphi$.

For $H\in\Ha$ we define
\begin{equation}\label{2.4}
\psi(H):= \min\{m(K,x):x\in H\}.
\end{equation}
The minimum is attained (a proof is given in Section 3 of \cite{HS19}).

\begin{lemma}\label{L1a}
For $u\in{\rm supp}\,\varphi$,
\begin{equation}\label{2.6}
h(K[\varphi,\delta],u) = \tau\Leftrightarrow \psi(H(u,\tau))=\delta.
\end{equation}
\end{lemma}

\begin{proof} 
For $H\in\Ha\setminus\Ha_{{\rm int}\,K}$, we state that
\begin{equation}\label{14.8.4a} 
\psi(H)\le \delta \Leftrightarrow H\cap K[\varphi,\delta]\not=\emptyset.
\end{equation}
In fact, if $H\cap K[\varphi,\delta]\not=\emptyset$, then $H$ contains a point $x$ with $m(K,x)\le\delta$, hence $\psi(H)\le\delta$. If $H\cap K[\varphi,\delta]=\emptyset$, then every $x\in H$ satisfies $m(K,x)>\delta$, and since $\psi(H)$ is an attained minimum, also $\psi(H)>\delta$.

Let $u\in{\rm supp}\,\varphi$. We state that for all $\tau$ with $K\subset {\rm int}\,H^-(u,\tau)$, the function $\tau\mapsto \psi(H(u,\tau))$ is strictly increasing. For the proof, let $K\subset {\rm int}\,H^-(u,\tau_i)$, $i=1,2$, and let $\tau_1<\tau_2$. Let $z\in H(u,\tau_2)$ be such that $m(K,z)=\psi(H(u,\tau_2))$. Choose $o\in K$ and let $x$ be the point in the intersection of $[o,z]$ and $H(u,\tau_1)$. Every hyperplane that separates $K$ and $x$ also separates $K$ and $z$. Every neighborhood of $u$ on $\Sd$ has positive $\varphi$-measure. There is a neighborhood $U$ of $u$ with the following property. For each $v\in U$, the values $\tau$ for which the hyperplane $H(v,\tau)$ separates $K$ and $x$ make up an interval $I_x$ of positive length, and the interval $I_z$ for which $H(v,\tau)$ with $\tau\in I_z$ separates $K$ and $z$ is strictly longer than $I_x$. Now it follows that $\psi(H(u,\tau_2))=m(K,z)>m(K,x)\ge \psi(H(u,\tau_1))$.

The value $h(K[\varphi,\delta],u)$ of the support function is the maximal value $\tau$ for which $H(u,\tau)\cap K[\varphi,\delta]\not=\emptyset$. Therefore, from (\ref{14.8.4a}) and the proved strong monotonicity, the assertion of the lemma follows.
\end{proof}

\section{Applications to the $K$-cell of a Poisson hyperplane process}\label{sec3}

First we recall two appearances of floating bodies in stochastic geometry. Let $K_{(n)}$ denote the convex hull of $n$ independent uniform random points in $K$ (a convex body with interior points).  B\'{a}r\'{a}ny and Larman \cite[Thm. 1]{BL88} showed that the expected volume $\bE V(K_{(n)})$ can be estimated  in terms of the floating body $K_\delta$, in the following way. There are constants $c_1$ and $c_2(d)$ (independent of $K$) such that
\begin{equation}\label{3.1}
c_1[V(K)-V(K_{1/n})] \le V(K)-\bE V(K_{(n)}) \le c_2(d) [V(K)-V(K_{1/n})].
\end{equation}
Another result concerns the set-valued expectation of an integrable random convex body $X$. By definition, this expectation, also called the {\em selection expectation} of $X$ and denoted by $\bE X$, is the closure of the set of all integrable selections of $X$; see Molchanov \cite{Mol05}, Definition 1.12. We need here only that the support functions satisfy 
$$ h(\bE X,u)= \bE h(X,u)\quad\mbox{for }u\in\Sd;$$
see \cite[Thm. 1.22]{Mol05}.

For the convex hull $K_{(n)}$ of $n$ independent uniform random points in a convex body $K$ and for the floating body $K_\delta$, B\'{a}r\'{a}ny and Vitale \cite{BV93} have shown the existence of constants $0<a<b<\infty$ such that for all $n\in\N$,
$$
K\setminus K_{a/n} \subseteq K\setminus\bE X_{(n)}\subseteq K\setminus K_{b/n},
$$
or equivalently,
\begin{equation}\label{3.2} 
h(K_{b/n},\cdot)\le  h(\bE X_{(n)},\cdot)\le h(K_{a/n},\cdot).
\end{equation}

Dually to convex hulls of random points in a convex body, we now consider intersections of random halfspaces containing a convex body. These halfspaces are bounded by hyperplanes of a Poisson hyperplane process, and we first describe this setting.

We assume that a stationary Poisson hyperplane process $\widehat X$ in $\R^d$ is given. (We refer to \cite{SW08}, in particular Section 3.2 for Poisson processes and Section 4.4 for processes of flats. As there, we identify a simple counting measure with its support, if convenient. So we write $H\in\widehat X$ for $\widehat X(\{H\})=1$.) The (locally finite) intensity measure $\widehat \Theta =\bE\widehat X$ can be written as
$$ \widehat\Theta(A) =\widehat\gamma\int_{\Sd}\int_{\R} {\mathbbm 1}_A(H(u,\tau))\,\D\tau\,\varphi(\D u)$$
for Borel sets $A\subset\Ha$. Here $\varphi$ is an even probability measure on the sphere $\Sd$, called the {\em spherical directional distribution} of $\widehat X$. The number $\widehat\gamma>0$ is the {\em intensity} of $\widehat X$.  The underlying probability space is denoted by $(\Omega,{\bf A},\bP)$, and the distribution of $\widehat X$ is given by 
$$
\bP\left(\widehat X(A)=k\right) = e^{-\widehat\Theta(A)}\frac{\widehat\Theta(A)^k}{k!}
$$
for $k\in\N_0$ and Borel sets $A\subset\Ha$ with $\widehat\Theta(A)<\infty$.

As above, we assume that the measure $\varphi$ is not concentrated on a great subsphere. With the measure $\nu$ defined by (\ref{2.1}), we have
$$ \widehat\Theta=\widehat\gamma\nu.$$
As before, we define $m(K,x)$ by (\ref{2.1b}) and and $\Phi(K)$ by (\ref{2.1a}). 

We assume now that a convex body $K$ with interior points is given. In order to emphasize the analogy to convex hulls of $n$ random points in $K$, we assume in the following that the intensity $\widehat\gamma$ of $\widehat X$ is a number $n\in\N$.

Aiming at `dual' results, we define the {\em $K$-cell} of the Poisson hyperplane process $\widehat X$ by
$$ Z_K^{(n)}:= \bigcap_{H\in \widehat X,\,H\cap K=\emptyset} H^-(K),$$
where $H^-(K)$ is the closed halfspace bounded by the hyperplane $H$ that contains $K$. Thus, $Z_K^{(n)}$ is a random polytope containing the convex body $K$. Replacing $K$ by $\{o\}$ (where $o$ is the origin of $\R^d$), we obtain the {\em zero cell} of $\widehat X$, denoted by $Z_o^{(n)}$. If we assume (without loss of generality) that $o\in K$, then the distribution of the $K$-cell is the conditional distribution of the zero cell, under the condition that the latter contains $K$,
\begin{equation}\label{3.4}
\bP\left(Z_K^{(n)}\in\cdot\right) = \bP\left(Z_o^{(n)}\in\cdot\mid K\subset Z_o^{(n)}\right).
\end{equation}
This follows from the independence properties of Poisson processes, since the sets $\Ha_K$ and $\Ha\setminus \Ha_K$ are disjoint.

First we prove a counterpart to (\ref{3.2}). If the support of the spherical directional distribution is the whole sphere, it shows that the set-valued expectations of the random polytopes $Z_K^{(n)}$ can be estimated in terms of the (non-random) separation bodies $K[\varphi,\delta]$ of $K$, for suitable parameters  $\delta$ depending on $n$.

\begin{theorem}\label{T3.2}
Suppose that ${\rm supp}\,\varphi=\Sd$. Then
\begin{eqnarray*}
e^{-1}[h(K[\varphi,1/n],\cdot)-h(K,\cdot)]
&\le&  h(\bE Z_K^{(n)},\cdot)-h(K,\cdot)\\
&\le& (1+e^{-1})[h(K[\varphi,1/n],\cdot)-h(K,\cdot)].
\end{eqnarray*}
\end{theorem}

\begin{proof}
We recall that $h(\bE Z_K^{(n)},\cdot)= \bE h(Z_K^{(n)},\cdot)$.
Let $u\in \Sd$ be given. We choose
$$ o\in H(K,u)\cap K,$$
so that $h(K,u)=0$ (clearly, $\bE h(Z_K^{(n)},u)- h(K,u)$ is independent of the choice of the origin).

For $t\ge 0$,
\begin{equation}\label{3.6}
h(Z_K^{(n)},u)\ge t\Leftrightarrow H(u,t)\cap Z_K^{(n)}\not=\emptyset.
\end{equation}
We have
\begin{equation}\label{3.7}
\bE h(Z_K^{(n)},u) = \int_\Omega h(Z_K^{(n)},u)\,\D\bP
= \int_0^\infty \bP\left(h(Z_K^{(n)},u)\ge t\right)\,\D t.
\end{equation}

In dealing with the left inequality of Theorem \ref{T3.2}, we make use of the definitions (\ref{2.1b}) and (\ref{2.4}) (keeping in mind that $\widehat \Theta=n\nu$). Let $t\ge 0$, and let $z\in H(u,t)$ be such that
$$ \psi(H(u,t)) = m(K,z).$$
If no hyperplane of $\widehat X$ separates $K$ and $z$, then $z\in Z_K^{(n)}$ and hence $H(u,t)\cap Z_K^{(n)}\not=\emptyset$. It follows that
\begin{eqnarray*}
\bP\left(h(Z_K^{(n)},u)\ge t\right) &=& \bP\left(H(u,t)\cap Z_K^{(n)}\not=\emptyset\right)\\
&\ge& \bP\left(\widehat X(\Ha_{K^z}\setminus\Ha_K)=0\right)\allowdisplaybreaks\\
&=& \exp\left[-\widehat\Theta(\Ha_{K^z}\setminus\Ha_K)\right]\allowdisplaybreaks\\
&=& e^{-nm(K,z)}\\
&=& e^{-n\psi(H(u,t))}.
\end{eqnarray*}
Using (\ref{3.7}), for any $\tau>0$ we get
\begin{eqnarray*}
\bE h(Z_K^{(n)},u)&=& \int_0^\infty \bP\left(h(Z_K^{(n)},u)\ge t\right)\,\D t\nonumber\\
&\ge& \int_0^\infty  e^{-n\psi(H(u,t))}\,\D t\nonumber\\
&\ge& \int_0^\tau  e^{-n\psi(H(u,t))}\,\D t\nonumber\\
&\ge& e^{-n\psi(H(u,\tau))}\cdot\tau.
\end{eqnarray*}
(We have used that $t\mapsto \psi(H(u,t))$ is increasing.) By the choice of the origin, this is equivalent to
\begin{equation}\label{3.8}
\bE h(Z_K^{(n)},u)-h(K,u)\ge e^{-n\psi(H(u,\tau))}\tau.
\end{equation}
We finish this later.

To deal with the right inequality of Theorem \ref{T3.2}, from (\ref{3.7}) and (\ref{3.6}) we deduce, for any $\tau\ge 0$, that
\begin{eqnarray}
\bE h(Z_K^{(n)},u) &=& \int_0^\infty \bP\left(H(u,t)\cap Z_K^{(n)}\not=\emptyset\right)\,\D t\nonumber\\
&\le& \tau + \int_{\tau}^\infty \bP\left(H(u,t)\cap Z_K^{(n)}\not=\emptyset\right)\,\D t.\label{3.9}
\end{eqnarray} 

Suppose that $t>\tau$, for a given $\tau>0$, is such that $H(u,t)\cap Z_K^{(n)}\not=\emptyset$. Then there is a point $x\in H(u,t)$ that belongs to $Z_K^{(n)}$, hence no hyperplane of $\widehat X$ strictly separates $K$ and $x$. Observing that the set of hyperplanes separating $K$ and $x$, but not strictly, has $\nu$-measure zero, we obtain
\begin{equation}\label{3.10} 
\bP\left(H(u,t)\cap Z_K^{(n)}\not=\emptyset\right) \le \bP\left(\widehat X(\Ha(K|x))=0\right)=e^{-\widehat\Theta(\Ha(K|x))}.
\end{equation}

Let 
$$ y=\frac{1}{s}x \quad\mbox{be such that}\quad y\in H(u,\tau),$$
then $s=t/\tau$ (by the choice of the origin). Let $U\subset\Sd$ be the set of all unit vectors $v$ for which some hyperplane $H(v,r)$ separates $K$ and $y$. For $v\in U$ we have $\langle x,v\rangle >\langle y,v\rangle>0$ and $h(K,v)\ge 0$, hence
$$ \frac{\langle x,v\rangle-h(K,v)}{\langle y,v\rangle-h(K,v)} \ge \frac{\langle x,v\rangle}{\langle y,v\rangle}=s.$$
If a hyperplane $H(v,r)$ separates $K$ and $y$, then the hyperplane $H(v,sr)$ separates $K$ and $x$. Hence, if all hyperplanes $H(v,r)$ with $r$ in some interval $I$ separate $K$ and $y$, then all hyperplanes $H(u,r)$ with $r$ in the interval $sI$ separate $K$ and $x$. It follows that
\begin{eqnarray*}
\widehat\Theta(\Ha(K|x)) &=& n\int_{\Sd}\int_0^\infty {\mathbbm 1}\{H(v,t)\in\Ha(K|x)\}\,\D t\,\varphi(\D v)\\
&\ge& n\int_{\Sd}s\int_0^\infty {\mathbbm 1}\{H(v,t)\in\Ha(K|y)\}\,\D t\,\varphi(\D v)\\
&=& s\widehat\Theta(\Ha(K|y))\\
&\ge& sn\psi(H(u,\tau)).
\end{eqnarray*}
Now it follows from (\ref{3.9}) and (\ref{3.10}) that
\begin{eqnarray}\label{3.11}
\bE h(Z_K^{(n)},u)-h(K,u) &\le& \tau+\int_\tau^\infty e^{-n\frac{t}{\tau}\psi((H(u,\tau))}\,\D t\nonumber\\
&=& \tau\left(1+\int_1^\infty e^{-nr\psi(H(u,\tau))}\D r\right).
\end{eqnarray}

Finally, we choose $\tau$ such that $\psi(H(u,\tau))=1/n$. Then
$$ h(K[\varphi,1/n],u)-h(K,u)=\tau$$
by (\ref{2.6}) and the choice of the origin. Therefore, (\ref{3.8}) and (\ref{3.11}) yield Theorem \ref{T3.2}. 
\end{proof}

As mentioned, duality interchanges points and hyperplanes, and hence should replace measures of points by measures of hyperplanes. In this sense, a conceptual dual to the volume of the random polytope $K_{(n)}$ would not be the volume of the random polytope $Z_K^{(n)}$, but its mean width. Therefore, the following theorem can be considered as the proper dual counterpart to (\ref{3.1}).

\begin{theorem}\label{T3.3}
Suppose that ${\rm supp}\,\varphi=\Sd$. Then
$$ e^{-1}[W(K[\varphi,1/n])-W(K,)]\le  \bE W(Z_K^{(n)})-W(K)\le(1+e^{-1})[W(K[\varphi,1/n])-W(K)].$$
\end{theorem}

\begin{proof}
Since
$$ W(K) = 2\int_{\Sd} h(K,u)\,\sigma(\D u),$$
the assertion follows from Theorem \ref{T3.2} by integration, using $\bE h(Z_K^{(n)},\cdot) = h(\bE Z_K^{(n)},\cdot)$ and Fubini's theorem.
\end{proof}

For the volume, we can prove an analogue of the left inequality of (\ref{3.1}). The proof of the right inequality in (\ref{3.1}) is more involved; it uses the `economic cap covering theorem' of B\'{a}r\'{a}ny and Larman \cite{BL88}. We do not know of a `dual' result, nor wether it would be helpful. 

\begin{theorem}\label{T3.1}
We have
$$e^{-1}[V(K[\varphi,1/n])-V(K)] \le \bE V(Z_K^{(n)}) -V(K).$$
\end{theorem}

\begin{proof}
The proof relies on formula (\ref{3.13}), for which we give a short proof. A related formula was proved by Kaltenbach \cite[Section 7]{Kal90}. We assume that $o\in {\rm int}\,K$. By (\ref{3.4}),
$$ \bE V(Z_K^{(n)})=\bE\left(V(Z_o^{(n)})\mid K\subset Z_o^{(n)}\right) =\bE\left(V(Z_o^{(n)}\mid \widehat X(\Ha_K)=0\right).$$
Recall that $[o,x]$ is the closed segment with endpoints $o$ and $x$. Since 
$$x\in{\rm int}\,Z_o^{(n)}\Leftrightarrow \widehat X(\Ha_{[o,x]})=0$$ 
and
$${\mathbbm 1}\{\widehat X(\Ha_K)=0\}{\mathbbm 1}\{\widehat X(\Ha_{[o,x]})=0\} = {\mathbbm 1}\{\widehat X(\Ha_{K^x})=0\},$$
we obtain (with $\lambda$ denoting Lebesgue measure in $\R^d$)
\begin{eqnarray*}
&& \bE V(Z_K^{(n)})\\
&& = \bP(\widehat X(\Ha_K)=0)^{-1} \bE\left({\mathbbm  1}\{\widehat X(\Ha_K)=0\}V(Z_o^{(n)})\right)\\
&& = e^{\widehat\Theta(\Ha_K)} \bE\int_{\R^d} {\mathbbm 1}\{\widehat X(\Ha_K)=0\}{\mathbbm 1}\{\widehat X(\Ha_{[o,x]}=0\}\,\lambda(\D x)\\
&& = e^{\widehat\Theta(\Ha_K)} \bE\int_{\R^d} {\mathbbm 1}\{\widehat X(\Ha_{K^x})=0\}\,\lambda(\D x)\\
&& = e^{\widehat\Theta(\Ha_K)} \left[\bE\int_K {\mathbbm 1}\{\widehat X(\Ha_{K^x})=0\}\,\lambda(\D x) + \bE\int_{\R^d\setminus K} {\mathbbm 1}\{\widehat X(\Ha_{K^x})=0\}\,\lambda(\D x)\right].
\end{eqnarray*}
Here the first expectation in the brackets is equal to
$$ e^{-\widehat\Theta(\Ha_K)} V(K),$$
and the second, by Fubini's theorem, is equal to
$$
\int_{\R^d\setminus K} e^{-\widehat\Theta(\Ha_{K^x})}\lambda(\D x).
$$
Thus, we obtain
\begin{equation}\label{3.13}
\bE V(Z_K^{(n)}) - V(K) = \int_{\R^d\setminus K} e^{-\widehat\Theta(\Ha_{K^x}\setminus \Ha_K)}\lambda(\D x).
\end{equation}
This yields, for any $\tau>0$,
\begin{eqnarray*}
\bE V(Z_K^{(n)}) - V(K) 
&=& \int_{\R^d\setminus K} e^{-nm(K,x)}\lambda(\D x)\\
&\ge & \int_{K[\varphi,\tau]\setminus K} e^{-nm(K,x)}\lambda(\D x)\\
&\ge& e^{-n\tau}(V(K[\varphi,\tau]\setminus K).
\end{eqnarray*}
The choice $\tau=1/n$ gives
$$ \bE V(Z_K^{(n)})-V(K) \ge e^{-1}V(K[\varphi,1/n]\setminus K)$$
and thus the assertion.
\end{proof}

\noindent Author's address:\\[2mm]Rolf Schneider\\Mathematisches Institut, Albert-Ludwigs-Universit{\"a}t\\D-79104 Freiburg i. Br., Germany\\E-mail: rolf.schneider@math.uni-freiburg.de


\begin{thebibliography}{10}
\bibitem{Bar89}  B\'ar\'any, I., Intrinsic volumes and $f$-vectors of random polytopes. {\em Math. Ann.} {\bf 285} (1989), 671--699.
\bibitem{Bar07}  B\'ar\'any, I., Random polytopes, convex bodies, and approximation. In: A. Baddeley, I. B\'{a}r\'{a}ny, R. Schneider, W. Weil, {\em Stochastic Geometry}, pp. 77--118, Lecture Notes in Math. {\bf 1892}, Springer, Berlin, 2007.
\bibitem{Bar08} B\'ar\'any, I., Random points and lattice points in convex bodies. {\em Bull. Amer. Math. Soc.} {\bf 45} (2008), 339--356.
\bibitem{BL88} B\'{a}r\'{a}ny I., Larman, D.G., Convex bodies, economic cap coverings, random polytopes.  {\em Mathematika} {\bf 35} (1988), 274--291.
\bibitem{BV93} B\'{a}r\'{a}ny, I., Vitale, R.A., Random convex hulls: floating bodies and expectations. {\em J. Approximation Th.} {\bf 75} (1993), 130--135.
\bibitem{BLW18} Besau, F., Ludwig, M., Werner, E., Weighted floating bodies and polytopal approximation. {\em Trans. Amer. Math. Soc.} {\bf 370} (2018), 7129--7148.
\bibitem{BSW18} Besau, F., Sch\"utt, C., Werner, E., Flag numbers and floating bodies. {\em Adv. Math.} {\bf 338} (2018), 912--952. 
\bibitem{BW17} Besau, F., Werner, E., The spherical convex floating body. {\em Adv. Math.} {\bf 301} (2016), 867--901.
\bibitem{BW18} Besau, F., Werner, E., The floating body in real space forms. {\em J. Differential Geom.} {\bf 110} (2018), 187--220.
\bibitem{BS10} B\"or\"oczky, K.J., Schneider, R., The mean width of circumscribed random polytopes. {\em Canad. Math. Bull.} {\bf 53} (2010), 614--628.
\bibitem{FR50} F\'{a}ry, I., R\'{e}dei, L., Der zentralsymmetrische Kern und die zentralsymmetrische H\"ulle von konvexen K\"orpern. {\em Math. Ann.} {\bf 122} (1950), 205--220.
\bibitem{Fre12} Fresen, D., The floating body and the hyperplane conjecture. {\em Arch. Math.} {\bf 98} (2012), 389--397.
\bibitem{Fre13} Fresen, D., A multivariate Gnedenko law of large numbers. {\em Ann. of Probab.} {\bf 41} (2013), 3051--3080.
\bibitem{HS19} Hug, D., Schneider, R., Poisson hyperplane processes and approximation of convex bodies. arXiv:1908.09498v1.
\bibitem{JW14} Jenkinson, J., Werner, E., Relative entropies for convex bodies. {\em Trans. Amer. Math. Soc.} {\bf366} (2014), 2889--2906.
\bibitem{Kal90} Kaltenbach, F.J., Asymptotisches Verhalten zuf\"alliger konvexer Polyeder. Doctoral Thesis, Albert-Ludwigs-Universit\"at, Freiburg, 1990.
\bibitem{LSW19} Li, B., Sch\"utt, C., Werner, E.M., Floating functions. {\em Israel J. Math.} {\bf 231} (2019), 181--210.
\bibitem{Mol05} Molchanov, I., {\em Theory of Random Sets.} Springer, London, 2005.
\bibitem{MW17} Mordhorst, O., Werner, E.M., Duality of floating and illumination bodies.\\ arXiv:1709.02424v1
\bibitem{MW17b} Mordhorst, O., Werner, E.M., Floating and illumination bodies for polytopes: duality results. {\em Discrete Anal.} 2019, Paper No. 11, 22 pp.
\bibitem{NSW19} Nagy, S., Sch\"utt, C., Werner, E., Halfspace depth and floating body. {\em Stat. Surv.} {\bf 13} (2019), 52--118.
\bibitem{Sch14} Schneider, R., {\em Convex Bodies: The Brunn--Minkowski Theory.} 2nd edn., Encyclopedia of Mathematics and Its Applications, vol. {\bf 151}, Cambridge University Press, Cambridge, 2014.
\bibitem{SW08} Schneider, R., Weil, W., {\em Stochastic and Integral Geometry.} Springer, Berlin, 2008.
\bibitem{Sch91} Sch\"{u}tt, C., The convex floating body and polyhedral approximation. \textit{Israel J. Math.} \textbf{73} (1991), 65--77.
\bibitem{Sch99} Sch\"{u}tt, C., Floating body, illumination body, and polytopal approximation. {\em Convex Geometric Analysis} (Berkeley, CA, 1996; K.M. Ball, V. Milman, eds.), pp. 203֭--229, Math. Sci. Res. Inst. Publ., {\bf 34}, Cambridge Univ. Press, Cambridge, 1999
\bibitem{Sch02} Sch\"utt, C., Best and random approximation of convex bodies by polytopes. {\em Rend. Circ. Mat. Palermo (2) Suppl.} {\bf 70}, part II (2002), 315--334.
\bibitem{SW90} Sch\"{u}tt, C., Werner, E., The convex floating body. \textit{Math. Scand.} \textbf{66} (1990), 275--290.
\bibitem{SW92} Sch\"{u}tt, C., Werner, E., The convex floating body of almost polygonal bodies. {\em Geom. Dedicata} {\bf 44} (1992), 169--188.
\bibitem{SW94} Sch\"{u}tt, C., Werner, E., Homothetic floating bodies. {\em Geom. Dedicata} {\bf 49} (1994), 335--348.
\bibitem{SW04} Sch\"{u}tt, C., Werner, E., Surface bodies and $p$-affine surface area. {\em Adv. Math.} {\bf 187} (2004), 98--145. 
\bibitem{Wer94} Werner, E., Illumination bodies and affine surface area. {\em Studia Math.} {\bf 110} (1994), 257--269.
\bibitem{Wer96} Werner, E., The illumination body of a simplex. {\em Discrete Comput. Geom.} {\bf 15} (1996), 297--306.
\bibitem{Wer99} Werner, E., A general construction for affine surface areas. {\em Studia Math.} {\bf 132} (1999), 227--238.
\bibitem{Wer02} Werner, E., The $p$-affine surface area and geometric interpretations. {\em Rend. Circ. Mat. Palermo (2) Suppl.} {\bf 70}, part II (2002), 367--382.
\bibitem{Wer06} Werner, E., Floating bodies and illumination bodies. In {\em Integral Geometry and Convexity: Proceedings of the International Conference, Wuhan, China, 2004}, (Grinberg, E.L., Li, S., Zhang, G., Zhou, J., eds.), pp. 129--140,  World Scientific, Hackensack, NJ, 2006.
\bibitem{Wer07} Werner, E., On $L_p$-affine surface areas. {\em Indiana Univ. Math. J.} {\bf 56} (2007), 2305--2323.
\bibitem{WY08} Werner, E., Ye, D., New $L_p$ affine isoperimetric inequalities. {\em Adv. Math.} {\bf 218} (2008), 762--780.
\bibitem{WY10} Werner, E., Ye, D., Inequalities for mixed $p$-affine surface areas. {\em Math. Ann.} {\bf 347} (2010), 703--737.
\bibitem{WY11} Werner, E., Ye, D., On the homothety conjecture. {\em Indiana Univ. Math. J.} {\bf 60} (2011), 1--20.
\end{thebibliography}
\end{document}